\documentclass[journal,onecolumn,12pt]{IEEEtran}
\usepackage{amsmath}
\usepackage{amsfonts}
\usepackage{amssymb}
\usepackage{amsthm}
\usepackage{bbm}
\usepackage{float}
\usepackage{makecell}
\usepackage{url}
\usepackage{amscd}
\usepackage{mathrsfs}
\usepackage{graphicx}
\usepackage[margin=0.5in]{geometry} 

\newcommand{\isep}{\mathrel{{.}\,{.}}\nobreak}

\usepackage{setspace}
\usepackage[OT2,T1]{fontenc}
\DeclareSymbolFont{cyrletters}{OT2}{wncyr}{m}{n}
\DeclareMathSymbol{\Sha}{\mathalpha}{cyrletters}{"58}

\newtheorem{theorem}{Theorem}

\newtheorem{lemma}[theorem]{Lemma}

\newcommand{\Fq}{\mathbb{F}_q}
\newcommand{\A}{\mathbf{A}}
\newcommand{\bC}{\mathbb{C}}

\newcommand{\SC}{\mathrm{SC}}
\newcommand{\tr}{\mathrm{Tr}}
\newcommand{\Tr}{{\bf Tr}}
\newcommand{\ep}{\varepsilon}
\newcommand{\E}{\mathbb{E}}
\newcommand{\G}{\mathcal{G}_n}
\newcommand{\M}{\mathcal{M}}
\newcommand{\C}{\mathcal{C}}
\newcommand{\Cb}{\C^\bot}
\newcommand{\D}{\mathcal{D}}
\newcommand{\F}{\mathbb{F}}

\renewcommand{\P}{\mathbb{P}}
\newcommand{\om}{\omega}
\newcommand{\db}{d^\bot}
\newcommand{\supp}{\mathrm{supp}}
\renewcommand{\l}{\ell}

\newcommand{\Gl}{G^{(\l)}}

\newcommand{\Gnl}{\G^{(\l)}}
\newcommand{\Phil}{\Phi_n^{(\l)}}

\newcommand{\rsc}{\varrho_\SC}
\newcommand{\Al}{A_\l}
\newcommand{\Yl}{Y_\l}
\newcommand{\Zl}{Z_\l}
\newcommand{\Ml}{M_n^{(\l)}}
\newcommand{\dl}{\Delta_\l}
\newcommand{\an}{\alpha_n}
\renewcommand{\i}{\mathrm{i}}
\newcommand{\weta}{\widetilde{\eta}}
\renewcommand{\d}{\mathrm{d}}
\newcommand{\hmu}{\hat{\mu}_n}
\newcommand{\hs}{\hat{s}_n}
\newcommand{\pa}{\partial}

\newcommand{\ml}{\mathbf{m}_{\l}}
\newcommand{\st}{\Gamma_{\tau}}
\newcommand{\stt}{\mathbf{S}_{\tau}}
\newcommand{\OpI}{\Omega_{p,I}}
\newcommand{\lt}{\mathbf{L}_{\tau,\ep}}
\newcommand{\lttt}{\mathbf{L}_\tau}
\newcommand{\I}{\mathcal{I}}

\title{Random Matrices from Linear Codes and the Convergence to Wigner's Semicircle Law}
 \author{Chin Hei Chan and Maosheng Xiong}

\begin{document}
\maketitle
\begin{abstract}
Recently we considered a class of random matrices obtained by choosing distinct codewords at random from linear codes over finite fields and proved that under some natural algebraic conditions their empirical spectral distribution converges to Wigner's semicircle law as the length of the codes goes to infinity. One of the conditions is that the dual distance of the codes is at least 5. In this paper, employing more advanced techniques related to Stieltjes transform, we show that the dual distance being at least 5 is sufficient to ensure the convergence, and the convergence rate is of the form $n^{-\beta}$ for some $0 < \beta < 1$, where $n$ is the length of the code.

\end{abstract}
\begin{keywords}
Group randomness, linear code, dual distance, empirical spectral measure, random matrix theory, Wigner's semicircle law.
\end{keywords}
\section{Introduction}\label{intro}
Random matrix theory is the study of matrices whose entries are random variables. Of particular interest is the study of eigenvalue statistics of random matrices such as the empirical spectral measure. It has been broadly investigated in a wide variety of areas, including statistics \cite{Wis}, number theory \cite{MET}, economics \cite{econ}, theoretical physics \cite{Wig} and communication theory \cite{TUL}.

Most of the matrix models considered in the literature were matrices whose entries have independent structures. In a series of work (\cite{Tarokh2,Tarokh3,OQBT}), initiated in \cite{Tarokh1}, the authors studied a class of matrices formed by choosing codewords at random from linear codes over finite fields and ultimately proved the convergence of the empirical spectral distribution of their Gram matrices to the Marchenko-Pastur law under the condition that the minimum Hamming distance of the dual codes is at least 5. This is the first result relating the randomness of matrices from linear codes to the algebraic properties of the underlying dual codes, and can be interpreted as a joint randomness test for sequences from linear codes. It implies in particular that sequences from linear codes with desired properties behave like random sequences from the view point of random matrix theory. This is called a ``group randomness'' property in \cite{Tarokh1} and may have many applications (see \cite{WigCode,WigM} from a different perspective). 

Recently we considered a distinct normalization of matrices obtained in a similar fashion from linear codes and proved the convergence of the empirical spectral distribution to the Wigner's semicircle law under some natural algebraic conditions of the underlying codes (see \cite{CSC}). This is also a group randomness property of linear codes. In this paper we explore this new phenomenon much further.

\subsection{Statement of Main Results}

To describe our results more precisely, we need some notation. Let $\mathscr{C}=\{\C_i : i \ge 1\}$ be a family of linear codes of length $n_i$ and dimension $k_i$ over the finite field $\Fq$ of $q$ elements ($\C_i$ is called an $[n_i,k_i]_q$ code for short), where $q$ is a prime power. The most interesting case is binary linear codes, corresponding to $q=2$. Denote by $\C_i^\bot$ the dual code of $\C_i$ and $d_i^\bot$ the Hamming distance of $\C_i^\bot$. $d_i^\bot$ is also called the \emph{dual distance} of $\C_i$.

The standard additive character of $\F_q$ extends component-wise to a natural mapping $\psi: \F_q^{n_i} \to (\mathbb{C}^*)^{n_i}$. For each $i$, we choose $p_i$ distinct codewords from $\C_i$ and apply the mapping $\psi$. Endowing with uniform probability on the choice of the $p_i$ codewords, this forms a probability space. Put the $p_i$ distinct sequences as the rows of a $p_i \times n_i$ random matrix $\Phi_{\C_i}$. Denote
\begin{equation}\label{Gram}
\mathcal{G}_{\C_i}=\frac{1}{n_i}\Phi_{\C_i}\Phi_{\C_i}^*,
\end{equation}
where $\Phi_{\C_i}^*$ is the conjugate transpose of the matrix $\Phi_{\C_i}$ and define
\begin{equation}\label{Mn}
M_{\C_i}=\sqrt\frac{n_i}{p_i}(\mathcal{G}_{\C_i}-I_{p_i}).
\end{equation}
Here $I_{p_i}$ is the $p_i \times p_i$ identity matrix.

For any $n \times n$ matrix $\A$ with eigenvalues $\lambda_1,\ldots,\lambda_n$, the \emph{spectral measure} of $\A$ is defined by
$$\mu_\A=\frac{1}{n}\sum_{j=1}^n \delta_{\lambda_j},$$
where $\delta_{\lambda}$ is the Dirac measure at the point $\lambda$. The \emph{empirical spectral distribution} of $\A$ is defined by
$$F_\A(x):=\int_{-\infty}^x \mu_\A(\text{d}x).$$

Our first main result is as follows:
\begin{theorem} \label{thm}
Suppose $p_i, \frac{n_i}{p_i} \to \infty$ simultaneously as $i \to \infty$. If $\db_i \geq 5$ for any $i$, then as $i \to \infty$, we have
    \begin{eqnarray} \label{1:eq1} \mu_{n_i}(\I) \to \rsc(\I) \quad \emph{ in Probability},\end{eqnarray}
    and the convergence is uniform for all intervals $\I \subset \mathbb{R}$. Here $\mu_{n_i}$ is the spectral measure of the matrix $M_{\C_i}$ and $\varrho_\SC$ is the probability measure of the semicircle law whose density function is given by
\begin{equation}\label{scpdf}
\d\rsc(x)=\frac{1}{2\pi}\sqrt{4-x^2}\mathbbm{1}_{[-2,2]}\,\d x,
\end{equation}
and $\mathbbm{1}_{[-2,2]}$ is the indicator function of the interval $[-2,2]$.
\end{theorem}

We remark that originally in \cite{CSC} the same convergence (\ref{1:eq1}) was proved with an extra condition that there is a fixed constant $c > 0$ independent of $i$ such that
\begin{equation} \label{1:srip}
|\langle v,v' \rangle| \leq c\sqrt{n_i}, \quad \mbox{ for any } v \ne v' \in \psi(\C_i).
\end{equation}
The condition (\ref{1:srip}) is natural as explained in \cite{CSC}, and when $q=2$, it is equivalent to
\[\left|\mathrm{wt}(\mathbf{c})-\frac{n_i}{2}\right| \le \frac{c}{2} \sqrt{n_i}, \quad \forall \mathbf{c} \in \C_i \setminus \{\mathbf{0}\},\]
where $\mathrm{wt}(\mathbf{c})$ is the Hamming weight of the codeword $\mathbf{c}$. It is interesting that this extra condition can be dropped. Now the result of Theorem \ref{thm} has the same strength as that of \cite{OQBT} where the condition $d_i^\bot \ge 5$ alone is sufficient to ensure the convergence. It shall be noted that similar to \cite{OQBT}, the condition $d_i^\bot \ge 5$ in Theorem \ref{thm} is optimal because if $d_i^\bot=4 \, \forall i$, then Conclusion (\ref{1:eq1}) is false for first-order binary Reed-Muller codes which have dual distance $4$.

Our second main result shows that the rate of convergence (\ref{1:eq1}) is fast with respect to the length of the codes.

\begin{theorem} \label{thm1-2}
Let $\C$ be an $[n,k]_q$ code with dual distance $d^\bot \ge 5$. For fixed constants $\gamma_1,\gamma_2 \in (0,1)$ and $c \ge 1$, suppose $p$ and $n$ satisfy
\[c^{-1} n^{\gamma_1} \leq p \leq c \, n^{\gamma_2}.\]
Then
\begin{equation}\label{SD}
\left|\mu_{n}(\I)-\rsc(\I)\right| \prec n^{-\beta}
\end{equation}
uniformly for all intervals $\I \subset \mathbb{R}$, where $\beta>0$ is given by
\begin{equation}\label{beta}
\beta:=\min\left\{\frac{\gamma_1}{4},\frac{1-\gamma_2}{8} \right\}.
\end{equation}
\end{theorem}

We remark that the symbol ``$\prec$'' in (\ref{SD}) is a standard ``stochastic domination'' notation in probability theory (see \cite{Local law} for details), which means that for any $\ep >0$ and any $D>0$, there is a quantity $N(\ep,D,c,\gamma_1,\gamma_2)$, such that whenever $n \geq N(\ep,D,c,\gamma_1,\gamma_2)$, we have
\begin{equation}\label{SD2}
\sup_\I \P\left[|\mu_{n}(\I)-\rsc(\I)| > n^{-\beta+\ep} \right] \leq n^{-D}.
\end{equation}
Here $\P$ is the probability within the space of picking $p$ distinct codewords from $\C$ and the supremum is taken over all intervals $\I \subset \mathbb{R}$. Since $\ep, D$ and $N(\ep,D,c,\gamma_1,\gamma_2)$ do not depend on $\C$, the supremum can be taken over all linear codes $\C$ of length $n$ over $\F_q$ with $\db \geq 5$.

We also remark that $\db \geq 5$ is a very mild restriction on linear codes $\C$, and there is an abundance of binary codes that satisfy this condition, for example, the Gold codes (\cite{Gold}), some families of BCH codes (see \cite{DIN1,DIN2}) and many families of cyclic and linear codes studied in the literature (see for example \cite{CHE,TAN}). Such binary linear codes can also be generated by almost perfect nonlinear (APN) functions \cite{Blon,POTT}, a special class of functions with important applications in cryptography. 

\subsection{Simulations}

We illustrate Theorems \ref{thm} and \ref{thm1-2} by numerical experiments. We focus on binary Gold codes augmented by the all-1 vector. It is known that binary Gold codes have length $n=2^m-1$, dimension $2m$ and dual distance 5. The augmented binary Gold codes has length $n$, dimension $2m+1$ and dual distance at least 5. Because of the presence of the all-1 vector, the condition (\ref{1:srip}) is not satisfied. For each triple $(m,n,p)$ in the set $\{(5,31,8), (7,127,20),(9,511,35),(11,2047,50)\}$, we randomly pick $p$ codewords from the augmented binary Gold code of length $n=2^m-1$ and form the corresponding matrix, from which we use {\bf Sage} to compute the eigenvalues and plot the empirical spectral distribution along with Wigner's distribution (see Figures \ref{fig1} to \ref{fig4} below). We do the above 10 times for each such triple $(m,n,p)$ and at each time, we find that the plots are almost the same as before: they are all very close to Wigner's semicircle law and as the length $n$ increases, they become less and less distinguishable.

In order to illustrate more clearly the shape of the eigenvalue distribution, we also plot a density graph, which is shown in Figure \ref{fig5}. This is based on picking $p=100$ codewords from a binary Gold code of length $n=32767=2^{15}-1$.

From (\ref{beta}) it is easy to see that $\beta \leq 1/12$ and the upper bound is achieved when $\gamma_1=\gamma_2=1/3$. It might be possible to improve this value $\beta$ and hence obtain a better convergence rate. From the simulation results, however, it is not clear to us what the optimal $\beta$ that one may expect is.

\begin{figure}[ht]
\begin{center}
\includegraphics[angle=0,width=1.0 \textwidth,height=0.2 \textheight]{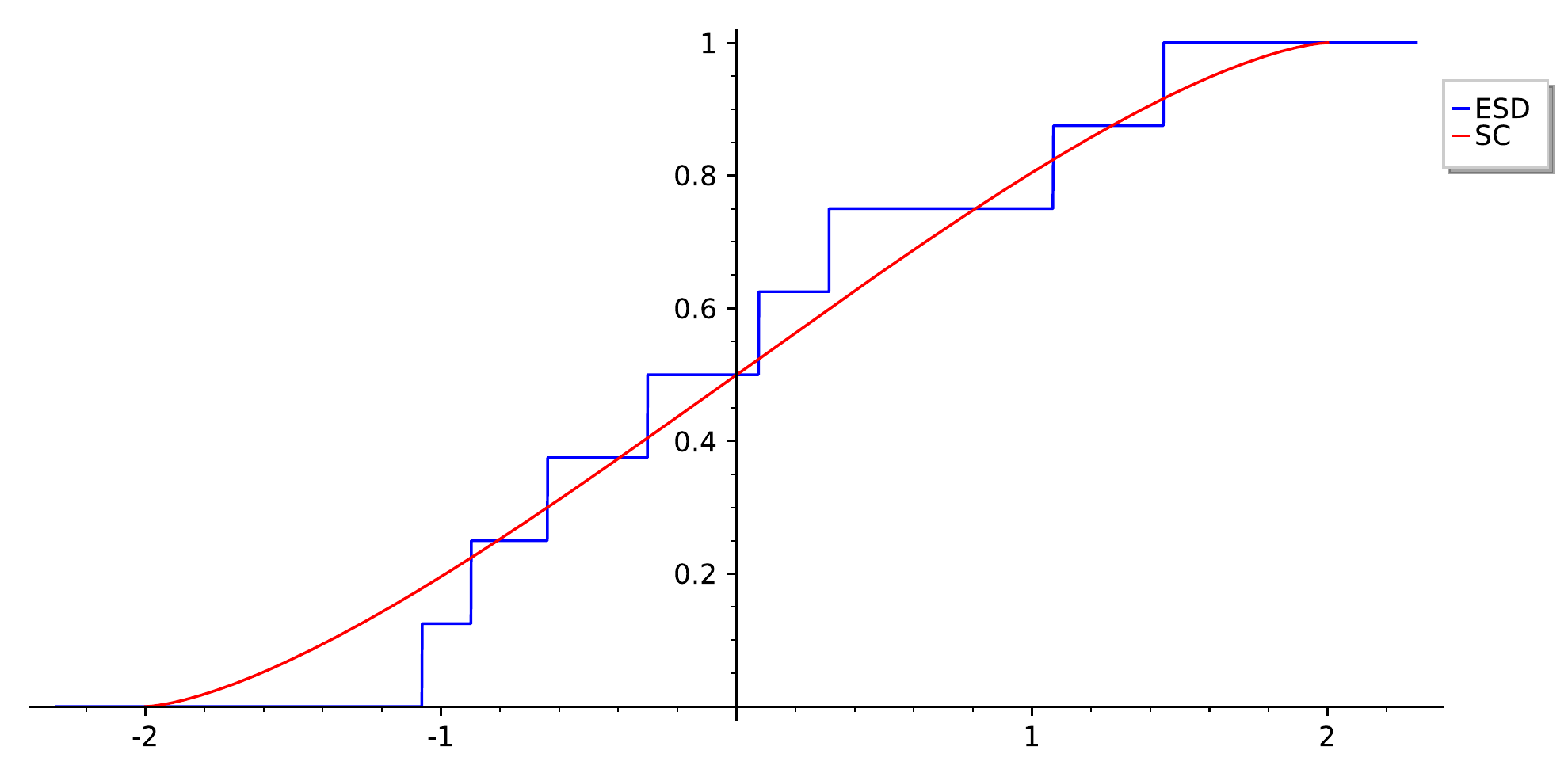}
\caption{Empirical spectral distribution (ESD) of $[31,11,12]$ augmented binary Gold code versus Wigner semicircle law (SC), with $p=8$}\label{fig1}
\end{center}
\end{figure}
\begin{figure}[ht]
\begin{center}
\includegraphics[angle=0,width=1.0 \textwidth,height=0.2 \textheight]{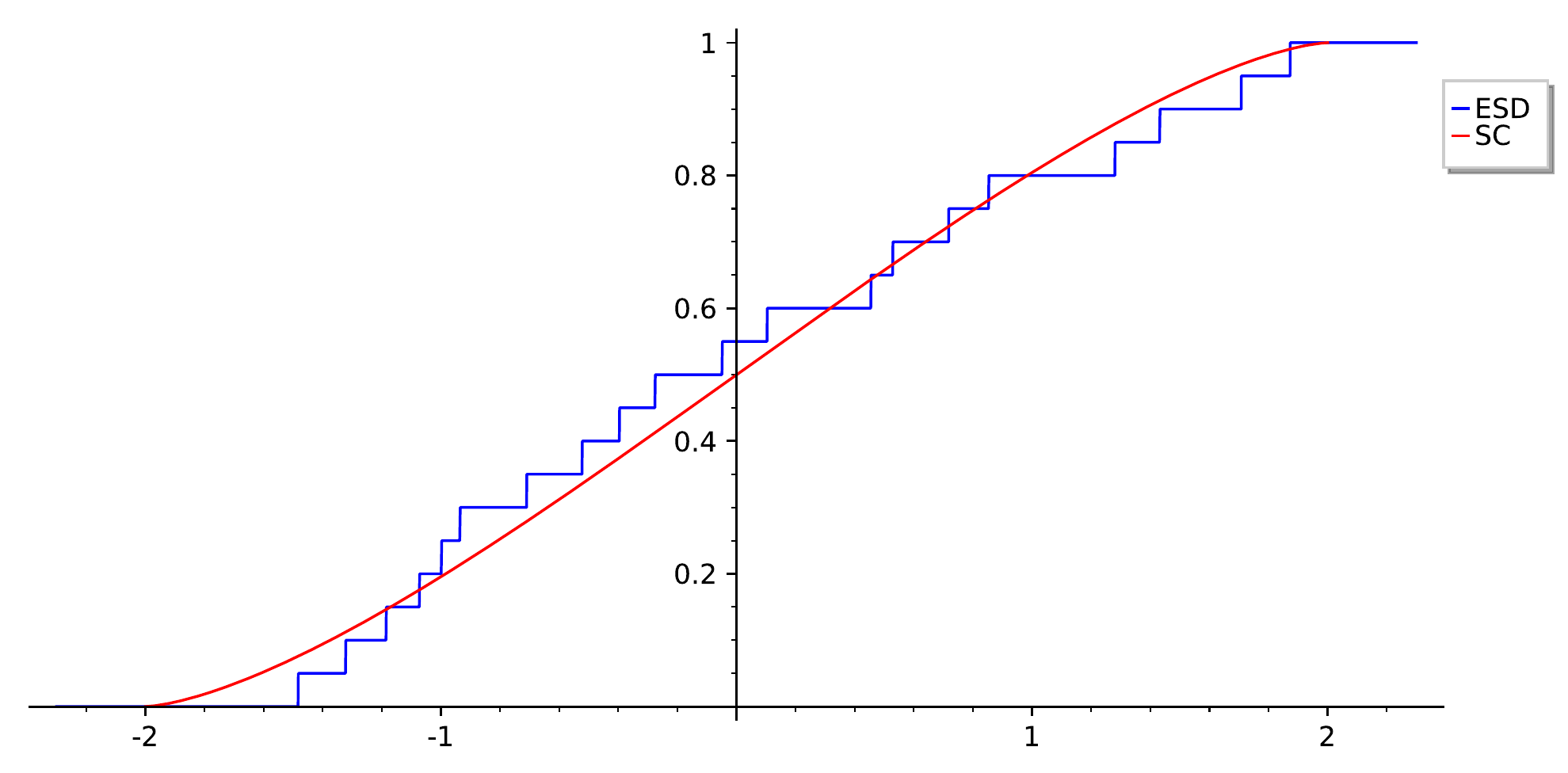}
\caption{Empirical spectral distribution (ESD) of $[127,15,56]$ augmented binary Gold code versus Wigner semicircle law (SC), with $p=20$}\label{fig2}
\end{center}
\end{figure}
\begin{figure}[ht]
\begin{center}
\includegraphics[angle=0,width=1.0 \textwidth,height=0.2 \textheight]{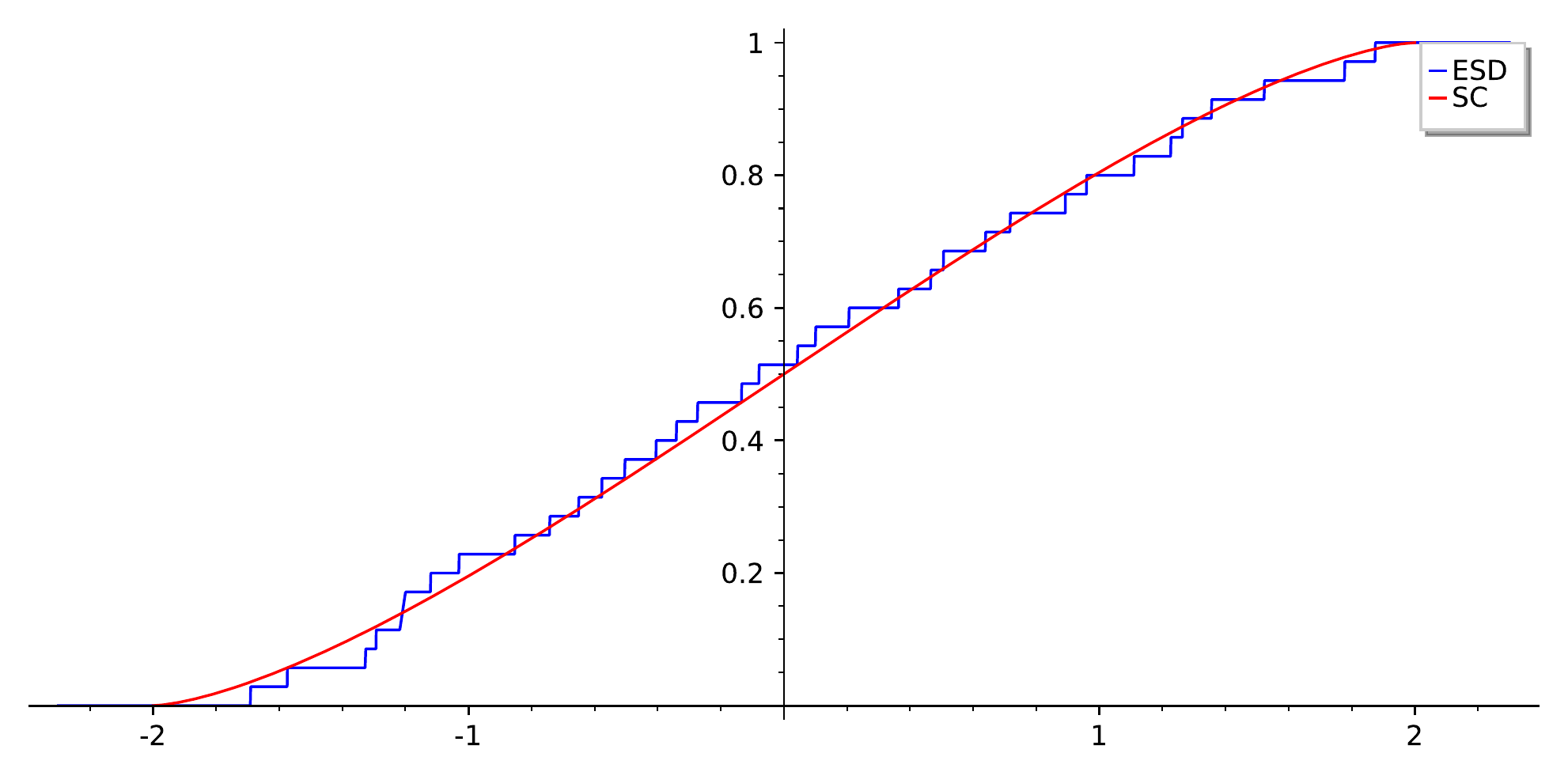}
\caption{Empirical spectral distribution (ESD) of $[511,19,240]$ augmented binary Gold code versus Wigner semicircle law (SC), with $p=35$}\label{fig3}
\end{center}
\end{figure}
\begin{figure}[ht]
\begin{center}
\includegraphics[angle=0,width=1.0 \textwidth,height=0.2 \textheight]{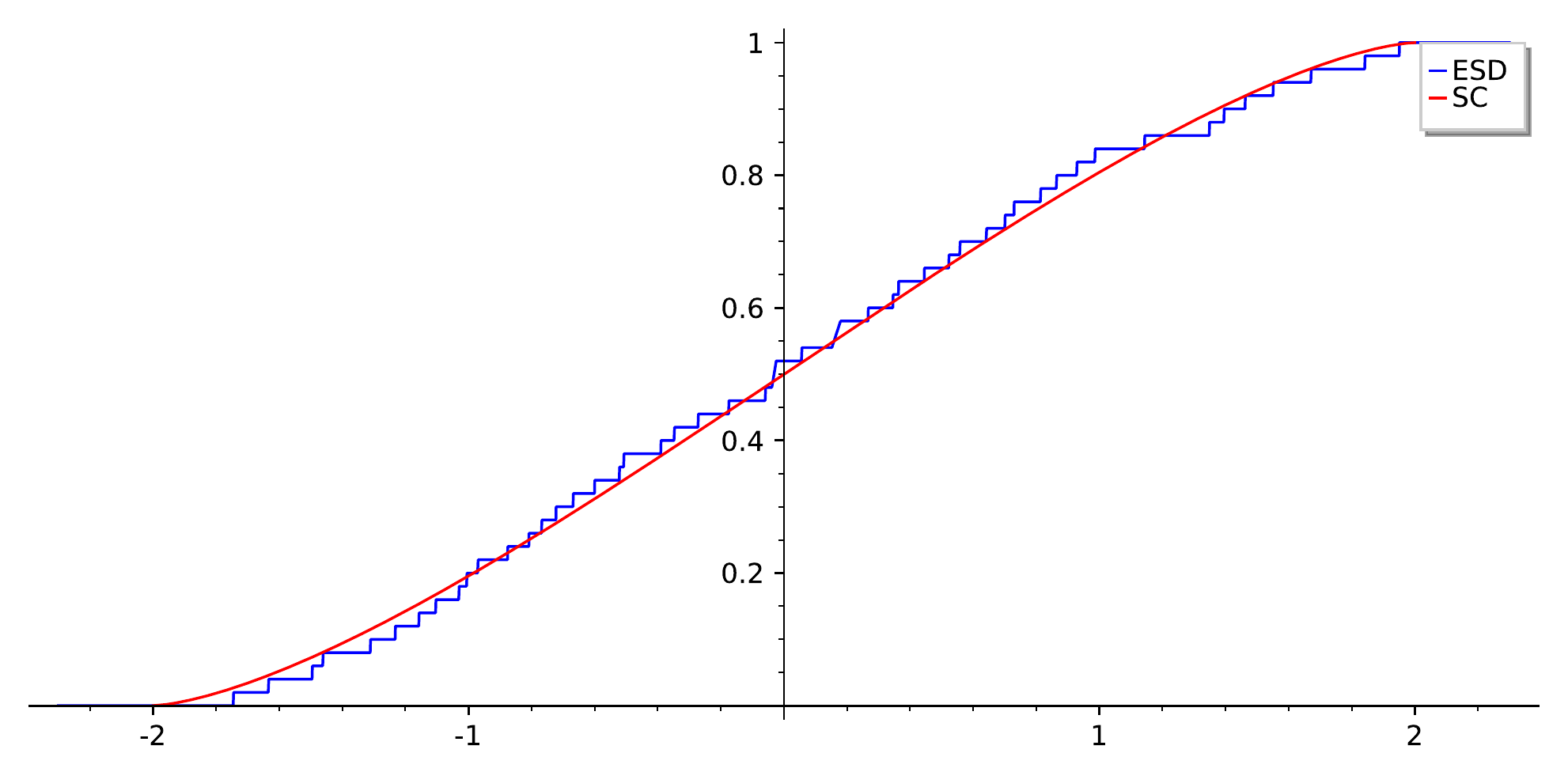}
\caption{Empirical spectral distribution (ESD) of $[2047,23,992]$ augmented binary Gold code versus Wigner semicircle law (SC), with $p=50$}\label{fig4}
\end{center}
\end{figure}
\begin{figure}[ht]
\begin{center}
\includegraphics[angle=0,width=1.0 \textwidth,height=0.2 \textheight]{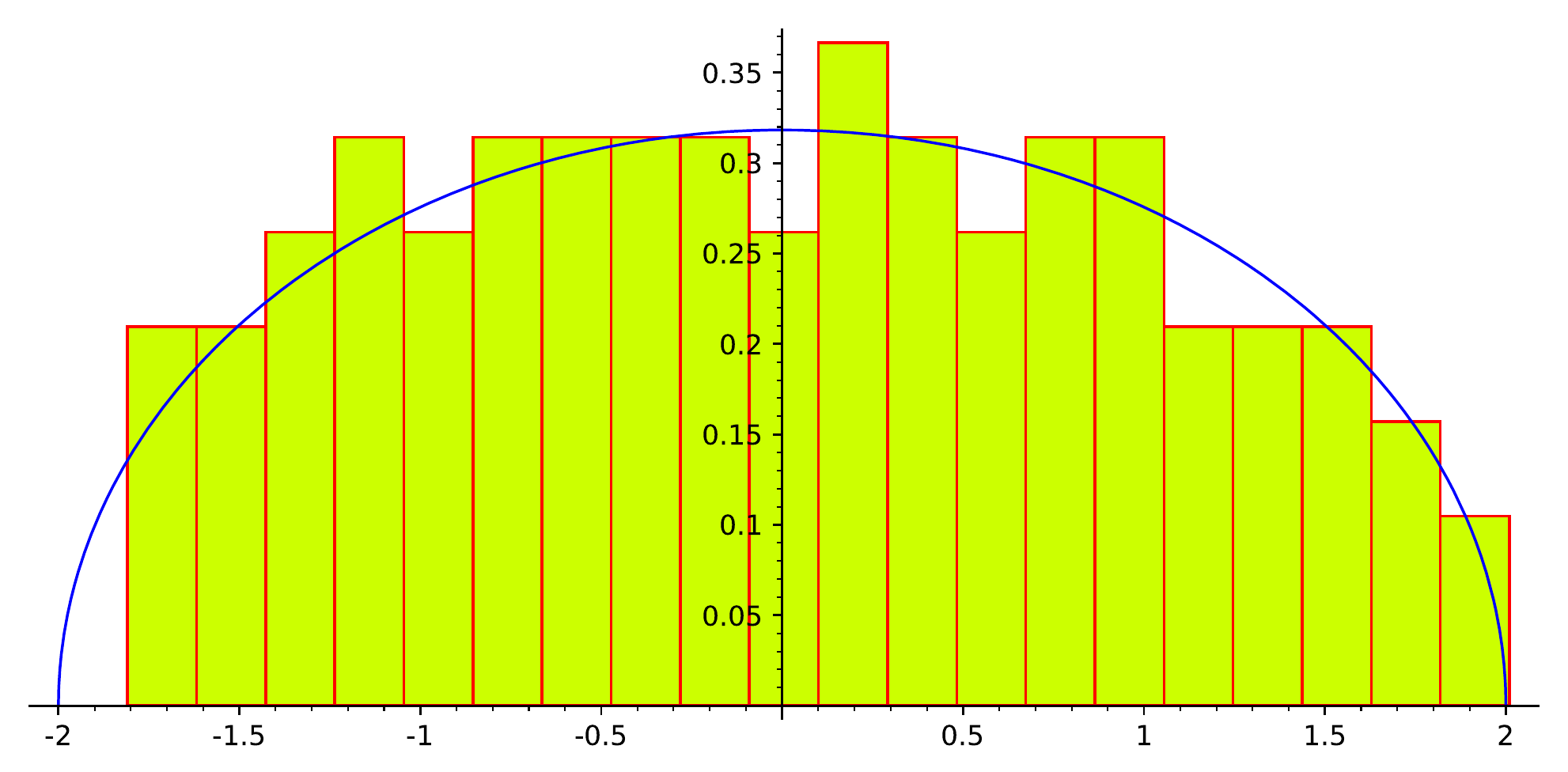}
\caption{Empirical spectral density of $[32767,30,16256]$ binary Gold code versus Wigner semicircle density, with $p=100$}\label{fig5}
\end{center}
\end{figure}

\subsection{Techniques and relation to previous work}

This paper strengthens \cite[Theorem 2]{CSC} on two fronts: in Theorem \ref{thm} we obtain the same convergence by removing the extra condition (\ref{1:srip}), and in Theorem \ref{thm1-2} we obtain a strong and explicit convergence rate with respect to the length of the code, and the results were supported by computer simulations. 

The main technique we use in this paper is the Stieltjes transform, a well-developed and standard tool in random matrix theory, and the method is essentially complex analysis. From the view point of random matrix theory, in \cite{BaiY,Bao,Xie} the authors have used Stieltjes transform to study similar matrix models with success, however, our matrices, arising from general linear codes over finite fields with dual distance 5, possess characteristics significantly different from \cite{BaiY,Bao,Xie}. With applications in mind, say, to generate pseudo-random matrices efficiently via linear codes, our matrices are more natural and interesting. None of the methods in previous works seem to apply directly to our setting. Instead we adopt methods from \cite{ConI,SA,Local law} and use a combination of ideas to obtain our final results.

Related to this paper, the authors in \cite{CTX} have used Stieltjes transform to obtain a strong convergence rate which is similar in nature to Theorem \ref{thm1-2} of this paper, hence extending the work \cite{OQBT}, and some of the arguments are similar.



The paper is organized as follows. In Section \ref{pre} we introduce Stieltjes transform and related formulas and lemmas which will play important roles later. The main ideas of proving Theorems \ref{thm} and \ref{thm1-2} share some similarity but technically speaking, they are quite involved, with the latter being even more so. To streamline the idea of the proofs, we assume a major technical statement (Theorem \ref{thm2}) from which we prove Theorems \ref{thm} and \ref{thm1-2} in Sections \ref{proof} and \ref{proof1-2} respectively. Finally we prove the required Theorem \ref{thm2} in Section \ref{proof2}. 


\section{Preliminaries}\label{pre}

\subsection{Linear codes over $\Fq$ of dual distance at least 5}

The standard additive character $\psi: \Fq \to \bC^*$ is given by
\begin{equation}\label{char}
\psi(a) = \zeta^{\tr(a)}, \quad \forall a \in \Fq,
\end{equation}
where $\tr$ is the absolute trace mapping from $\Fq$ to its prime subfield $\F_r$ of order $r$ and $\zeta=\exp(2\pi\sqrt{-1}/r)$ is a (complex) primitive $r$-th root of unity. In particular when $q=r=2$, then $\zeta=-1$ and $\psi(a)=(-1)^a$ for $a \in \mathbb{F}_2$. It is known that $\psi$ satisfies the following orthogonality relation:
\begin{equation} \label{2:orth} \frac{1}{q}\sum_{x \in \Fq} \psi(ax)= \left\{
\begin{array}{lll}
1 &:& \mbox{ if } a = 0;\\
0 &:& \mbox{ if } a \in \Fq \setminus \{0\}.
\end{array}\right.
\end{equation}

Let $\C$ be an $[n,k]_q$ linear code with dual distance $\db \ge 5$. By the sphere-packing bound \cite[Theorem 1.12.1]{FECC}, we have
$$\#\Cb =q^{n-k} \leq \frac{q^n}{1+n(q-1)+\binom{n}{2}(q-1)^2} =O\left(\frac{q^n}{n^2}\right),$$
here the implied constant in the big O-notation depends only on $q$. From this we can obtain
\begin{eqnarray} \label{le:db}
\frac{n^2}{q^k}=O(1).
\end{eqnarray}
Since $\C$ is linear, the orthogonal relation (\ref{2:orth}) further implies that for any $\mathbf{a} \in \F_q^n$, we have
\begin{equation}\label{lem}
\frac{1}{\#\C}\sum_{\mathbf{c} \in \C}\psi(\mathbf{a}\cdot \mathbf{c})=\left\{\begin{array}{lll}
1 &:& \mbox{ if } \mathbf{a} \in \Cb,\\
0 &:& \mbox{ if } \mathbf{a} \notin \Cb.
\end{array}\right.
\end{equation}
Here $\mathbf{a}\cdot\mathbf{c}$ is the usual inner product between the vectors $\mathbf{a}$ and $\mathbf{c}$ in $\F_q^n$.

\subsection{Stieltjes Transform}
In this section we recall some basic knowledge of Stieltjes transform. Interested readers may refer to \cite[Chapter B.2]{SA} for more details. Stieltjes transform can be defined for any real function of bounded variation. For the case of interest to us, however, we confine ourselves to functions arising from probability theory.

Let $\mu$ be a probability measure and let $F$ be the corresponding cumulative distribution function. The Stieltjes transform of $F$ or $\mu$ is defined by
$$s(z):=\int_{-\infty}^\infty \frac{\d F(x)}{x-z}=\int_{-\infty}^\infty \frac{\mu(\d x)}{x-z}, $$
where $z$ is a complex variable taking values in $\mathbb{C}^+:=\{z \in \mathbb{C}: \Im z > 0\}$, the upper half complex plane. Here $\Im z$ is the imaginary part of $z$.

It is known that $s(z)$ is well-defined for all $z \in \mathbb{C}^+$ and is well-behaved, satisfying the following properties:
\begin{itemize}
\item[(i).] $s(z) \in \mathbb{C}^+$ for any $z \in \mathbb{C}^+$;

\item[(ii).] $s(z)$ is analytic in $\mathbb{C}^+$ and
\begin{eqnarray} \label{2:lips} \left|\frac{\d s(z)}{\d z}\right| \leq \int_{-\infty}^\infty \frac{\mu(\d x)}{|x-z|^2} \le \frac{1}{\eta^2},\end{eqnarray}
where $\eta=\Im z>0$;

\item[(iii).] the probability measure $\mu$ can be recovered from the Stieltjes transform $s(z)$ via the inverse formula (see \cite{SA}):
\begin{equation}\label{2:inverse}
\mu((x_1,x_2])=F(x_2)-F(x_1)=\lim_{\eta \to 0^+}\frac{1}{\pi}\int_{x_1}^{x_2} \Im(s(E+\i\eta)) \, \d E;
\end{equation}

\item[(iv).] the convergence of Stieltjes transforms is equivalent to the convergence of the underlying probability measures (see for example \cite[Theorem B.9]{SA}).
\end{itemize}

\subsection{Resolvent Identities and Formulas for Green function entries}

Let $M$ be a Hermitian $p \times p$ matrix whose $(j,k)$-th entry is $M_{jk}$. Denote by $G$ the Green function of $M$, that is,
\[G:=G(z)=(M-zI_p)^{-1},\]
where $z \in \mathbb{C}^{+}$. The $(j,k)$-th entry of $G$ is $G_{jk}$.

Given any subset $T \subset [1\isep p]:=\{1,2,\cdots,p\}$, let $M^{(T)}$ be the $p \times p$ matrix whose $(j,k)$-th entry is given by $(M^{(T)})_{jk}:=\mathbbm{1}_{j,k \notin T}M_{jk}$. In addition, let $G^{(T)}$ be the Green function of $M^{(T)}$, that is,
\[ G^{(T)}:=G^{(T)}(z)=(M^{(T)}-zI_p)^{-1}.\]

When $T$ is a singleton, say $\{\l\}$, it is common to further abbreviate the notation $G^{(\{\l\})}$ as $G^{(\l)}$, and similar for other matrices.

Let $\ml$ denote the $\l$-th column of $M$. For $z \in \mathbb{C}^{+}$ and any $\l \in [1\isep p] \setminus T$, we have the \emph{Schur complement formula} (see \cite{SA, Local law})
\begin{equation}\label{diagonal}
\frac{1}{G_{\l\l}^{(T)}}=M_{\l\l}-z-\ml^*G^{(T\l)}\ml,
\end{equation}
where $G^{(T\l)}:=G^{(T \cup \{\l\})}$ and $\ml^*$ is the conjugate transpose of $\ml$.


We also have the following eigenvalue interlacing property (see \cite{SA, Local law})
\begin{equation}\label{Interlacing} |\Tr G^{(T)}(z)-\Tr G(z)| \leq C\eta^{-1},
\end{equation}
where $z=E+\i\eta \in \mathbb{C}^{+}$, $\Tr$ is the trace function, and $C$ is a constant depending only on the set $T$.
\subsection{Stieltjes Transform of the Semicircle Law}
The Stieltjes transform $s_\SC$ of the semicircle distribution given in (\ref{scpdf}) can be computed as (see \cite{SA})
\begin{equation}\label{SSC}
s_\SC(z)=\frac{-z+\sqrt{z^2-4}}{2}.
\end{equation}
Here and throughout this paper, we always pick the complex square root $\sqrt{\cdot}$ to be the one with positive imaginary part.

It is well-known that $s_\SC(z)$ is the unique function that satisfies the equation
\begin{equation}\label{SSC2}
u(z)=\frac{1}{-z-u(z)}
\end{equation}
such that $\Im u(z) > 0$ whenever $\eta:=\Im z > 0$.

\subsection{Convergence of Stieltjes Transform in Probability}
In order to bound the convergence rate of a random Stieltjes transform in probability, we need the following well-known McDiarmid's lemma from probability theory (see \cite[Lemma F.3]{Local law}).

\begin{lemma}[McDiarmid]\label{leMic}
Let $X_1,\cdots,X_p$ be independent random variables taking values in the spaces $E_1,\cdots,E_p$ respectively. Let
$$f: E_1 \times \cdots E_p \to \mathbb{R}$$
be a measurable function and define the random variable $Y=f(X_1,\cdots,X_p)$. Define, for each $k \in [1\isep p]$,
\begin{equation}\label{ck}
c_k:=\sup |f(x_1,\cdots,x_{k-1},y,x_{k+1},\cdots,x_p)-f(x_1,\cdots,x_{k-1},z,x_{k+1},\cdots,x_p)|,
\end{equation}
where the supremum is taken over all $x_j \in E_j$ for $j \neq k$ and $y,z \in E_k$. Then for any $\ep > 0$, we have
\begin{equation}\label{Mic}
\P\left(|Y-\E Y| \geq \ep \right) \leq 2\exp\left(-\frac{2\ep^2}{c_1^2+\cdots+c_p^2}\right).
\end{equation}
\end{lemma}
We will need the following concentration inequality. We remark that a very similar concentration inequality was proved (see \cite[Lemma F.4]{Local law}). Here for the sake of completeness, we provide a detailed proof.

\begin{lemma}\label{deviation}
Let $\M$ be a $p \times n$ random matrix with independent rows, define $S=(n/p)^{1/2}(\M\M^*-I_p)$. Let $s(z)$ be the Stieltjes transform of the empirical spectral distribution of $S$. Then for any $\ep > 0$ and $z =E+i \eta \in \mathbb{C}^+$,
$$\P\left(|s(z)-\E s(z)| \geq \ep \right) \leq 2\exp\left(-\frac{p\eta^2\ep^2}{8}\right).$$
\end{lemma}
\begin{proof}[Proof of Lemma \ref{deviation}]
Applying Lemma \ref{leMic}, we take $X_j$ to be the $j$-th row of $\M$ and the function $f$ to be the Stieltjes transform $s$. Note that the $(j,k)$-th entry of $S$ is a linear function of the inner product of the $j$-th and $k$-th rows of $\M$. Hence changing one row of $\M$ only gives an additive perturbation of $S$ of rank at most two. Applying the resolvent identity \cite[(2.3)]{Local law}, we see that the Green function is also only affected by an additive perturbation by a matrix of rank at most two and operator norm at most $2\eta^{-1}$. Therefore the quantities $c_k$ in (\ref{ck}) can be bounded by
$$c_k \leq \frac{4}{p\eta}.$$
Then the required result follows directly from inserting the above bound to (\ref{Mic}).
\end{proof}

\section{Proof of Theorem \ref{thm}}\label{proof}

Throughout the paper, let $\C$ be an $[n,k]_q$ linear code over $\Fq$. We always assume that its dual distance satisfies $\db \ge 5$. Denote $N=q^k$. The standard additive character on $\Fq$ extends component-wise to a natural mapping $\psi: \Fq^{n} \to \bC^{n}$. Define $\D=\psi(\C)$.

\subsection{Problem set-up}

Theorems \ref{thm} and \ref{thm1-2} are for random matrices in the probability space $\OpI$ of choosing $p$ distinct elements uniformly from $\D$. Denote by $\D^p$ the probability space of choosing $p$ elements from $\D$ independently and uniformly. Because $d^\bot \ge 5$, from (\ref{le:db}) we have
$$\frac{\#\D^p}{\#\OpI}=\frac{N^p}{N(N-1)(N-2)\cdots(N-p+1)}=1+O\left(\frac{p^2}{N}\right) \to 1,$$
as $n,p \to \infty$. Thus to prove Theorems \ref{thm} and \ref{thm1-2}, it is equivalent to consider the larger probability space $\D^p$. This will simplify the proofs.

Now let $\Phi_n$ be a $p \times n$ random matrix whose rows are picked from $\D$ uniformly and independently. Denote by $\E$ the expectation with respect to the probability space $\D^p$. We may assume that $p:=p(n)$ is a function of $n$ such that $p,n/p \to \infty$ as $n \to \infty$.

Let
\begin{eqnarray} \label{3:gn} \G=\frac{1}{n}\Phi_n\Phi_n^*, \quad M_n=\sqrt{\frac{n}{p}}(\G-I_p).\end{eqnarray}
Let $\mu_n$ be the empirical spectral measure of $M_n$ and let $s_{M_n}(z)$ be its Stieltjes transform, that is,
$$s_{M_n}(z)=\frac{1}{p}\sum_{j=1}^p \frac{1}{\lambda_j-z}=\frac{1}{p}\Tr G. $$
Here $\lambda_1,\cdots,\lambda_p$ are the eigenvalues of the matrix $M_n$, and $G:=G(z)$ is the Green function of $M_n$ given by
\[ G(z)=(M_n-zI_p)^{-1}.\]
Note that the Stieltjes transform $s_{M_n}(z)$ is itself a random variable in the space $\D^p$. We define
\begin{equation}\label{EStietjes}
s_n(z):=\E s_{M_n}(z)=\frac{1}{p}\E \Tr G.
\end{equation}
Throughout the paper, the complex value $z \in \mathbb{C}^+$ is always written as
\[z=E+ \i \eta, \quad \mbox{ where } E, \eta \in \mathbb{R} \mbox{ and } \eta >0. \]
For a fixed constant $\tau \in (0,1)$, we define
\begin{equation}\label{st}
\st:=\bigg\{z=E+\i\eta: |E| \leq \tau^{-1}, 0< \eta \leq \tau^{-1}
\bigg\}.
\end{equation}
Now we assume a result about the expected Stieltjes transform $s_n(z)$.
\begin{theorem}\label{thm2}
For any $z \in \st$, we write
\begin{eqnarray} \label{3:snz} s_n(z)=\frac{1}{-z-s_n(z)+\Delta(z)}. \end{eqnarray}
Then we have
\[ \Delta(z)=O_{\tau}\left(\eta^{-3} \left(p^{-1}+\sqrt{p/n}\right)\right).\]
\end{theorem}

We emphasize here that this is one of the major technical results in this paper and the proof is a little complicated. This is the only result in the paper that is directly related to the properties of linear codes. It requires $\db \geq 5$ but not the extra condition (\ref{1:srip}) used in \cite{CSC}. To streamline the presentation, here we assume Theorem \ref{thm2}, then Theorem \ref{thm} can be proved easily. The proof of Theorem \ref{thm2} is postponed to Section \ref{proof2}.

\subsection{Proof of Theorem \ref{thm}}\label{sec}

By properties of the Stieltjes transform (see \cite[Theorem B.9]{SA}), to prove Theorem \ref{thm}, it is equivalent to prove the following statement: \emph{For any $\ep>0$, we have}
\begin{eqnarray} \label{3:conP} \P\left(\exists z \in \mathbb{C}^+ \mbox{ such that } \left|s_{M_n}(z) - s_\SC(z) \right| \ge \ep \right) \to 0 \quad \mbox{ as } n \to \infty.
\end{eqnarray}
We prove Statement (\ref{3:conP}) in several steps.

First, we fix an arbitrary value $z \in \mathbb{C}^+$. The quadratic equation (\ref{3:snz}) has two solutions
$$s_n^{\pm}(z)=\frac{-(z-\Delta)\pm \sqrt{(z-\Delta)^2-4}}{2}.$$
As $n \to \infty$, from Theorem \ref{thm2} we have $\Delta(z) \to 0$, so $z-\Delta \in \mathbb{C}^+$ for large enough $n$. Since $s_n(z),s_\SC(z) \in \mathbb{C}^+$, we see that
\begin{equation}\label{3:snz2}
s_n(z)=s_n^+(z)=s_\SC(z-\Delta).
\end{equation}
Then by the continuity of $s_\SC$ and by taking $n \to \infty$, we obtain
\begin{equation}\label{3:limit}
 s_n(z) \to s_\SC(z).
\end{equation}
Moreover, by Lemma \ref{deviation}, for any fixed $\ep >0$, as $n \to \infty$, we have
\[\P\left(\left|s_{M_n}(z)-s_n(z)\right| \geq \ep\right) \to 0.\]
This and (\ref{3:limit}) immediately imply
\begin{equation}\label{3:prob}
\P\left(\left|s_{M_n}(z)-s_\SC(z)\right| \geq \ep\right) \to 0.
\end{equation}

Noting that (\ref{3:prob}) holds for any fixed $z \in \mathbb{C}^+$ and any $\ep>0$, so to prove (\ref{3:conP}), in the next step we need to show that the convergence is ``uniform'' for all $z \in \mathbb{C}^+$. To do this, we adopt a simple lattice argument.

For any $\tau,\ep \in (0,1)$, define the sets
$$\st':=\st \cap \{z=E+\i \eta: \eta \geq \tau\}$$
and
$$\lt:=\st' \cap \left\{z=\frac{\tau^2\ep}{4}(a+\i b): (a,b) \in \mathbb{Z}^2\right\}.$$
It is easy to see that $\lt \neq \emptyset$ and
$$\#\lt = O_\tau\left(\tau^{-4}\ep^{-2}\right) < \infty.$$
For any fixed $z \in \mathbb{C}^+$, define
$\Xi_{n,\ep}(z)$ to be the event
$$\left\{|s_{M_n}(z)-s_\SC(z)| < \ep \right\}.$$
By (\ref{3:prob}), for any $\delta > 0$, there is an $N(z,\tau,\ep,\delta)$ such that
$$n > N(z,\tau,\ep,\delta) \implies \P\left(\Xi_{n,\frac{\ep}{2}}(z)^{\bf c}\right) < \frac{\delta}{\#\lt}.$$
Here the set $\Xi_{n,\frac{\ep}{2}}(z)^{\bf c}$ denotes the complement of the event $\Xi_{n,\frac{\ep}{2}}(z)$. Then for any $n$ such that
\[ n > N(\tau,\ep,\delta):=\max_{z \in \lt} N(z,\tau,\ep,\delta),\]
we have
$$\P\left(\left(\bigcap_{z \in \lt} \Xi_{n,\frac{\ep}{2}}(z)\right)^{\bf c}\right)=\P\left(\bigcup_{z \in \lt} \Xi_{n,\frac{\ep}{2}}(z)^{\bf c}\right) < \delta. $$

Finally we consider the event $\bigcap_{z \in \lt} \Xi_{n,\frac{\ep}{2}}(z)$, that is,
\[ \left|s_{M_n}(z')-s_\SC(z') \right|< \frac{\ep}{2} \quad \forall z' \in \lt.\]
Recall from (\ref{2:lips}) that the Stieltjes transforms $s_{M_n}(z)$ and $s_\SC(z)$ are both $\tau^{-2}$-Lipschitz on the set $\st'$, and for any $z \in \st'$, we can find one $z' \in \lt$ such that
$$|z-z'| \leq \frac{\tau^2\ep}{4}. $$
So for this $z \in \st'$ we have
\begin{align*}
|s_{M_n}(z)-s_\SC(z)| &\leq |s_{M_n}(z)-s_{M_n}(z')|+|s_{M_n}(z')-s_\SC(z')|+|s_\SC(z')-s_\SC(z)|\\
&< \tau^{-2}|z-z'|+\frac{\ep}{2}+\tau^{-2}|z-z'|\\
& \le \ep.
\end{align*}
This means that
$$\bigcap_{z \in \lt} \Xi_{n,\frac{\ep}{2}}(z) \subset \bigcap_{z \in \st'} \Xi_{n,\ep}(z).$$
Therefore
$$\P\left(\left(\bigcap_{z \in \st'} \Xi_{n,\ep}(z)\right)^{\bf c}\right) \leq \P\left(\left(\bigcap_{z \in \lt} \Xi_{n,\frac{\ep}{2}}(z)\right)^{\bf c}\right) < \delta$$
for any $n > N(\tau,\ep,\delta)$.

Hence for any $\tau, \ep \in (0,1)$, we have
\begin{eqnarray*} \P\left(\exists z \in \st' \mbox{ such that } \left|s_{M_n}(z) - s_\SC(z) \right| \ge \ep \right) \to 0 \quad \mbox{ as } n \to \infty.
\end{eqnarray*}
Taking the limit $\tau \to 0^+$, we obtain the desired Statement (\ref{3:conP}). This completes the proof of Theorem \ref{thm}.

\section{Proof of Theorem \ref{thm1-2}}\label{proof1-2}
Now for fixed constants $c > 1$ and $\gamma_1,\gamma_2 \in (0,1)$, let us assume
\[ c^{-1}n^{\gamma_1} \leq p \leq cn^{\gamma_2}.\]
Similar in proving Theorem \ref{thm} in the previous section, here we assume Theorem \ref{thm2}. Then the main idea of proving Theorem \ref{thm1-2} is to provide a refined and quantitative version of Statement (\ref{3:conP}), so in each step of the proofs, we need to keep track of all the varying parameters as $n \to \infty$. 

First, the upper bound for $\Delta(z)$ in Theorem \ref{thm2} can be simplified as
\[ \Delta(z)=O_{c,\tau}\left(n^{-4\beta}\eta^{-3}\right), \]
where the constant $\beta>0$ is explicitly given in (\ref{beta}).

Let us define
\begin{eqnarray*}
\stt:=\st \bigcap \left\{z=E+\i \eta: \eta \ge n^{-\beta+\tau}\right\}.
\end{eqnarray*}
From now on, $C_{c,\tau}$ denotes some positive constant depending only on $c$ and $\tau$ whose value may vary at each occurrence. We can estimate the difference $|s_n(z)-s_\SC(z)|$ as follows.

\begin{lemma}\label{EStieltjes}
For any $z \in \stt$, we have
$$|s_n(z)-s_\SC(z)|=O_{c,\tau}\left(n^{-4\beta}\eta^{-4}\right).$$
\end{lemma}
\begin{proof}[Proof of Lemma \ref{EStieltjes}]
First, for large enough $n$, noting that
$$\Im(z-\Delta)\geq \eta-C_{c,\tau}n^{-4\beta}\eta^{-3} > 0,$$
we see that Equation (\ref{3:snz2}) holds for all $z \in \stt$. More precisely, we have
$$\Im(z-\Delta) > C_{c,\tau}\eta.$$
By using the fact $\left|\frac{\d s_\SC(z)}{\d z}\right| \leq \eta^{-1}$ which can be easily checked from (\ref{SSC}), we conclude that
$$|s_n(z)-s_\SC(z)|=\left|s_\SC(z-\Delta)-s_\SC(z) \right| \leq C_{c,\tau}\eta^{-1}|\Delta|\leq C_{c,\tau}n^{-4\beta}\eta^{-4}.$$
Then Lemma \ref{EStieltjes} is proved.
\end{proof}
Next we estimate the term $|s_{M_n}(z)-s_\SC(z)|$. An $n$-dependent event $\Xi$ is said to hold \emph{with high probability} if for any $D>0$, there is a quantity $N=N(D)>0$ such that $\P(\Xi) \geq 1-n^{-D}$ for any $n > N$.

\begin{theorem}\label{Stieltjes}
We have, with high probability,
$$|s_{M_n}(z)-s_\SC(z)| \leq n^\tau(n^{-\beta}+n^{-4\beta}\eta^{-4}) \quad \forall z \in \stt.$$
\end{theorem}
\begin{proof}[Proof of Theorem \ref{Stieltjes}]
By the concentration inequality given in Lemma \ref{deviation}, we have
\begin{equation}\label{deviation1}
\P\left(\left|s_{M_n}(z)-s_n(z)\right| \geq n^{\frac{\tau}{2}-\beta}\right)\leq 2\exp\left(-\frac{n^{\gamma_1-4\beta+3\tau}}{8c}\right)\leq 2\exp\left(-\frac{n^{3\tau}}{8c}\right).
\end{equation}
Noting that the inequality (\ref{deviation1}) holds for any fixed $z \in \stt$. In order to prove Theorem \ref{Stieltjes}, we need an upper bound which is uniform for all $z \in \stt$. We apply a lattice argument again.

Let
$$\lttt:=\stt \cap \left\{z=n^{-3 \beta} (a+\i b): (a,b) \in \mathbb{Z}^2 \right\}.$$
Note that the set $\lttt \ne \emptyset$ and
$$\#\lttt \leq C_\tau n^{6\beta}.$$
Also, for any $z \in \stt$ and $\ep > 0$, define $\mathcal{E}_{n,\ep}(z)$ to be the event
$$ \left\{|s_{M_n}(z)-s_n(z)| \leq n^{\ep-\beta} \right\},$$
and $\mathcal{E}_{n,\ep}(z)^{\bf c}$ the complement. Then (\ref{deviation1}) can be rewritten as
$$\P\left(\mathcal{E}_{n,\frac{\tau}{2}}(z)^{\bf c} \right) \leq 2\exp\left(-\frac{n^{3\tau}}{8c}\right).$$
So we have
\begin{eqnarray}
\P\left(\left(\bigcap_{z \in \lttt} \mathcal{E}_{n,\frac{\tau}{2}}(z)\right)^{\bf c}\right)&=&\P\left(\bigcup_{z \in \lttt} \mathcal{E}_{n,\frac{\tau}{2}}(z)^{\bf c}\right) \nonumber \\
&\leq & C_\tau n^{6\beta}\exp\left(-\frac{n^{3\tau}}{8c}\right) \leq n^{-D} \label{deviation2}
\end{eqnarray}
for any $D > 0$ and $n > N(c,\gamma_1,\gamma_2,\tau,D)$.

Finally we consider the event $\bigcap_{z \in \lttt} \mathcal{E}_{n,\frac{\tau}{2}}(z)$,
that is,
\[|s_{M_n}(z')-s_n(z')| \le n^{\frac{\tau}{2}-\beta} \quad \forall z' \in \lttt.\]
Noting that for any $z \in \stt$, there is $z' \in \lttt$ such that
$$|z-z'| \leq n^{-3\beta}$$
and that $s_{M_n}(z)$ and $s_n(z)$ are both $n^{2\beta}$-Lipschitz on $\stt$, we obtain, for any $z \in \stt$,
\begin{align*}
\left|s_{M_n}(z)-s_n(z) \right| &\leq |s_{M_n}(z)-s_{M_n}(z')|+|s_{M_n}(z')-s_n(z')|+|s_n(z')-s_n(z)|\\
&< 2n^{2\beta}|z-z'|+n^{\frac{\tau}{2}-\beta} \leq n^{\tau-\beta}.
\end{align*}
This means that
$$\bigcap_{z \in \lttt} \mathcal{E}_{n,\frac{\tau}{2}}(z) \subset \bigcap_{z \in \stt} \mathcal{E}_{n,\tau}(z).$$
Hence by (\ref{deviation2}) we have
$$\P\left(\bigcap_{z \in \stt} \mathcal{E}_{n,\tau}(z)\right) \geq \P\left(\bigcap_{z \in \lttt} \mathcal{E}_{n,\frac{\tau}{2}}(z)\right) \geq 1-n^{-D}$$
for all $n > N(c,\gamma_1,\gamma_2,\tau,D)$.

Combining the above inequality with Lemma \ref{EStieltjes} completes the proof of Theorem \ref{Stieltjes}.
\end{proof}

\begin{proof}[Proof of Theorem \ref{thm1-2}]
As a standard application of the Helffer-Sj\"{o}strand formula via complex analysis, Theorem \ref{thm1-2} can be derived directly from Theorem \ref{Stieltjes}. This is quite well-known, and the computation is routine. Interested readers may refer to \cite[Section 8]{Local law} for a very similar analysis. We omit the details. 
\end{proof}

\section{Proof of Theorem \ref{thm2}}\label{proof2}
In this section we give a detailed proof of Theorem \ref{thm2}, where the condition that $\db \geq 5$ plays an important role.

Recall from the beginning of Section \ref{proof} that $\C$ is a linear code of length $n$ over $\F_q$ with $\db \geq 5$, $\psi$ is the standard additive character on $\F_q$, extended component-wisely to $\F_q^n$, $\D=\psi(\C)$, and $\Phi_n$ is a $p \times n$ random matrix whose rows are selected uniformly and independently from $\D$. This makes $\D^p$ a probability space, on which we use $\E$ to denote the expectation. Let $\G$ and $M_n$ be defined as in (\ref{3:gn}). Since all the entries of $\Phi_n$ are roots of unity, the diagonal entries of $M_n$ are all zero.


Let $x_{jk}$ be the $(j,k)$-th entry of $\Phi_n$. The following properties of $x_{jk}$, while very simple, depend crucially on the condition that $d^\bot \ge 5$.

\begin{lemma}\label{cor}
For any $\l \in [1\isep p]$, we have

(a) $\E(x_{\l j}\overline{x}_{\l k})=0$ if $j \neq k$;

(b) $\E(x_{\l j}x_{\l t}\overline{x}_{\l k}\overline{x}_{\l s})=0$ if the indices $j,t,k,s$ do not come in pairs; If the indices come in pairs, then $|\E(x_{\l j}x_{\l t}\overline{x}_{\l k}\overline{x}_{\l s})| \leq 1$.
\end{lemma}
\begin{proof}[Proof of Lemma \ref{cor}]
(a) It is easy to see that
\[\E(x_{\l j}\overline{x}_{\l k})=\frac{1}{N}\sum_{\mathbf{c} \in \C}\psi(c_j-c_k)=\frac{1}{N}\sum_{\mathbf{c} \in \C}\psi(\mathbf{a}_1 \cdot \mathbf{c}),\]
where $\mathbf{c}=(c_1,c_2,\cdots,c_n) \in \C$ and $\mathbf{a}_1=(0, \cdots, 0, 1, 0 \cdots, 0, -1, 0 \cdots, 0) \in \F_q^n$. Here in $\mathbf{a}_1$ the 1 and $-1$ appear at the $j$-th and $k$-th entries respectively. Since $\db \geq 5$, we have $\mathbf{a}_1 \notin \Cb$, and the desired result follows directly from (\ref{lem}).

(b) It is easy to see that
\[\E(x_{\l j}x_{\l t}\overline{x}_{\l k}\overline{x}_{\l s})=\frac{1}{N}\sum_{\mathbf{c} \in \C}\psi(c_j+c_t-c_k-c_s)=\frac{1}{N}\sum_{\mathbf{c} \in \C}\psi(\mathbf{a}_2 \cdot \mathbf{c}),\]
where the vector $\mathbf{a}_2 \in \F_q^n$ is formed from the all-zero vector by adding $1$s to the $j$-th and $t$-th entries and then adding $-1$s from the $k$-th and $s$-th entries. If the indices $j,t,k,s$ do not come in pairs, then $0 \ne \mathrm{wt}(\mathbf{a}_2) \le 4$. Since $\db \geq 5$, we have $\E(x_{\l j}x_{\l t}\overline{x}_{\l k}\overline{x}_{\l s})=0$ by (\ref{lem}). The second statement of \emph{(b)} is trivial since $|x_{ij}|=1$ for any $i,j$.
\end{proof}

For any $\l \in [1 \isep p]$, let $\Phil$ be the $p \times n$ matrix obtained from $\Phi_n$ by changing the whole $\l$-th row to {\bf 0}. Define
\[ \Gnl:=\frac{1}{n}\Phil{\Phil}^*, \quad \Ml:=\sqrt{\frac{n}{p}}\left(\Gnl-I_p\right).\]
Denote by $\om(\l)$ the $\l$-th row of $\Phi_n$, and $\ml$ the $\l$-th column of $M_n$. It is easy to see that
\begin{eqnarray} \label{4:ml} \ml=\frac{1}{\sqrt{pn}}\Phil\om(\l)^*.\end{eqnarray}
Let
\[ G:=G(z)=\left(M_n-z I_p\right)^{-1}, \quad  \Gl:=\Gl(z)=\left(\Ml-z I_p\right)^{-1}\]
be the Green functions of $M_n$ and $\Ml$ respectively for the complex variable $z \in \mathbb{C}^+$.

For the Green function $G$, we start with the resolvent identity (\ref{diagonal}) for $T=\emptyset$. Using (\ref{4:ml}), we can express the third term on the right side of (\ref{diagonal}) as
\begin{align*}
\ml^*\Gl\ml&=\frac{1}{pn}\om(\l)\Phi_n^{(\l)*}\Gl\Phil\om(\l)^* \nonumber\\
&=\frac{1}{pn}\Tr \left({\Phil}^*\Gl\Phil\om(\l)^*\om(\l) \right).
\end{align*}
By the identity
$$(\om(\l)^*\om(\l))_{jk}=\delta_{jk}+(1-\delta_{jk})x_{\l j}\overline{x}_{\l k},$$
the right hand side can be further expressed as
\begin{align*}
\frac{1}{pn}\Tr \left({\Phil}^*\Gl\Phil\right)+\Zl
=\frac{1}{p}\Tr \left(\Gl\Gnl\right)+\Zl,
\end{align*}
where
\begin{equation}\label{Zl}
\Zl=\sum_{j \neq k} a_{jk} x_{\l j}\overline{x}_{\l k}.
\end{equation}
Here the indices $j,k$ vary in $[1\isep n]$ and $a_{jk}$'s are the $(jk)$-th entry of the $n \times n$ matrix $(a_{jk})$ given by
\begin{equation}\label{ajk}
(a_{jk})=\frac{1}{pn}{\Phil}^*\Gl\Phil.
\end{equation}
Hence the resolvent identity (\ref{diagonal}) yields
\begin{align*}
\frac{1}{G_{\l\l}}&=M_{\l\l}-z-\frac{1}{p}\Tr \left(\Gl\Gnl\right)-\Zl \\
&=-z-\frac{1}{p}\Tr \Gl\left(\sqrt{\frac{p}{n}}\Ml+I_p\right)-\Zl.
\end{align*}
Expanding the second term on the right, we obtain
\begin{align}
\frac{1}{G_{\l\l}}=-z-s_n(z)+\Yl,\label{diagonal2}
\end{align}
where
\begin{equation}\label{Yl}
\Yl=s_n(z)-\sqrt\frac{p}{n}-\left(\frac{1}{p}+\frac{z}{\sqrt{pn}}\right)\Tr \Gl-\Zl.
\end{equation}

\subsection{Estimates of $\Zl$ and $\Yl$}

The random variables $Z_{\l}$ and $Y_{\l}$ depend on the complex value $z =E+\i \eta \in \mathbb{C}^+$. For any fixed constant $\tau>0$, recall $\st$ defined in (\ref{st}). Throughout this section we always assume $z \in \st$. 

\begin{lemma}\label{Zl2}
Let $z \in \st$. Then for any $\l \in [1\isep p]$, we have

(a) $\E^{(\l)}\Zl=\E\Zl=0$. Here $\E^{(\l)}$ is the conditional expectation given $\{x_{jk}: j \ne \l\}$;

(b) $\E|\Zl|^2=O_{\tau}(p^{-1}\eta^{-2})$.
\end{lemma}
\begin{proof}[Proof of Lemma \ref{Zl2}]
(a) Since the rows of $\Phi_n$ are independent, the entries $a_{jk}$ as defined in (\ref{ajk}) are independent with $x_{\l j}$ and $x_{\l k}$. Hence from the definition of $\Zl$ in (\ref{Zl}) and statement (a) of Lemma \ref{cor}, we have
$$\E^{(\l)}\Zl=\sum_{j \neq k}a_{jk}\E(x_{\l j}\overline{x}_{\l k})=0.$$
The proof of the result on $\E\Zl$ is similar by replacing $a_{jk}$ with $\E a_{jk}$.

(b) Expanding $|\Zl|^2$ and taking expectation $\E$ inside, noting that the rows of $\Phi_n$ are independent, we have
\begin{align*}
\E|\Zl|^2&=\E\left|\sum_{j \neq k}a_{jk}x_{\l j}\overline{x}_{\l k}\right|^2\\
&=\sum_{\substack{j\neq k\\s\neq t}}\E(a_{jk}\overline{a}_{st})\E(x_{\l j}x_{\l t}\overline{x}_{\l k}\overline{x}_{\l s}).
\end{align*}
Since $\db \geq 5$, by using statement (b) of Lemma \ref{cor}, we find
$$\E|\Zl|^2 \leq C\sum_{j,k} \E|a_{jk}|^2=C\E \Tr ((a_{jk})(a_{jk})^*),$$
where $C$ is an absolute constant which may be different in each appearance. Using the definition of $(a_{jk})$ in (\ref{ajk}) we have
\begin{align*}
\E|\Zl|^2
&\leq\frac{C}{p^2}\E \Tr \left(\frac{1}{n^2}{\Phil}^*\Gl\Phil{\Phil}^*\Gl(\bar{z})\Phil\right)\\
&=\frac{C}{p^2}\E \Tr [(\Ml-z)^{-1}\Gnl(\Ml-\bar{z})^{-1}\Gnl]\\
&=\frac{C}{p^2}\E \sum_{j=1}^{p} \frac{\left(\sqrt\frac{p}{n}\lambda_j^{(\l)}+1\right)^2}{(\lambda_j^{(\l)}-z)(\lambda_j^{(\l)}-\bar{z})}.
\end{align*}
Expanding the terms on the right, we can easily obtain
\begin{align*}
\E|\Zl|^2
&\leq \frac{C}{p^2}\E \sum_{j=1}^{p}\left(\frac{p}{n}\left(1+\frac{|z|^2}{|\lambda_j^{(\l)}-z|^2}\right)+\frac{1}{|\lambda_j^{(\l)}-z|^2}\right)\\
&\leq C\left(\frac{1}{n}+\frac{|z|^2}{n\eta^2}+\frac{1}{p\eta^2}\right)\leq \frac{C_{\tau}}{p\eta^2}.
\end{align*}
Here $\lambda_j^{(\l)} \in \mathbb{R} (1 \leq j \leq p)$ are the eigenvalues of $\Ml$, and $C_{\tau}$ is a positive constant depending only on $\tau$ whose value may vary in each occurrence.
\end{proof}
The above estimations lead to the following estimations of $\Yl$.
\begin{lemma}\label{Yl2}
Let $z \in \st$. Then for any $\l \in [1\isep p]$, we have

(a) $\E \Yl=O_{\tau}\left(\eta^{-1}(p^{-1}+(p/n)^{\frac{1}{2}})\right)$;

(b) $\E|\Yl|^2=O_{\tau}\left(\eta^{-2}(p^{-1}+p/n)\right)$.
\end{lemma}
\begin{proof}[Proof of Lemma \ref{Yl2}]
(a) Taking expectation on $\Yl$ in (\ref{Yl}) and noting that $\E \Zl=0$, we get
\begin{align*}
\E \Yl &=\frac{1}{p}\E(\Tr G-\Tr \Gl)-\sqrt\frac{p}{n}-\frac{z}{\sqrt{pn}}\E\Tr\Gl.
\end{align*}
By the eigenvalue interlacing property in (\ref{Interlacing}) and the trivial bound $|G_{jj}^{(\l)}| \leq \eta^{-1}$, we get
$$|\E \Yl| \leq \frac{C}{p\eta}+\sqrt\frac{p}{n}+\sqrt\frac{p}{n}\frac{|z|}{\eta} \leq \frac{C}{p\eta}+\frac{C_{\tau}}{\eta}\sqrt\frac{p}{n}.$$
(b) We split $\E|\Yl|^2$ as
\begin{equation}\label{Yl3}
\E|\Yl|^2=\E|\Yl-\E\Yl|^2+|\E\Yl|^2=V_1+V_2+|\E\Yl|^2,
\end{equation}
where
$$V_1=\E|\Yl-\E^{(\l)}\Yl|^2, \quad V_2=\E|\E^{(\l)}\Yl-\E\Yl|^2. $$
We first estimate $V_1$. Using (a) of Lemma \ref{Zl2}, we see that
$$\Yl-\E^{(\l)}\Yl=-\Zl+\E^{(\l)}\Zl=-\Zl$$
Hence by (b) of Lemma \ref{Zl2} we obtain
\begin{equation}\label{V1}
V_1=\E|\Zl|^2=O_{\tau}(p^{-1}\eta^{-2}).
\end{equation}
Next we estimate $V_2$. Again by Lemma \ref{Zl2} we have
\begin{align*}
\E^{(\l)}\Yl-\E\Yl
&=-\left(\frac{1}{p}+\frac{z}{\sqrt{pn}}\right)(\Tr \Gl-\E\Tr \Gl).
\end{align*}
So we have
\begin{align}\label{V2}
V_2&=\left|\frac{1}{p}+\frac{z}{\sqrt{pn}}\right|^2\E|\Tr \Gl-\E\Tr \Gl|^2\nonumber\\
&=\left|\frac{1}{p}+\frac{z}{\sqrt{pn}}\right|^2\sum_{m \neq \l}\E\left|\E^{(T_{m-1})}\Tr \Gl-\E^{(T_m)}\Tr \Gl\right|^2.
\end{align}
Here we denote $T_0:=\emptyset$ and $T_m:=[1\isep m]$ for any $m \in [1\isep p]$, and for any subset $T \subset [1\isep p]$, we denote $\E^{(T)}$ to be the conditional expectation given $\{x_{jk}: j \notin T\}$. The second equality follows from applying successively the law of total variance to the rows of $\Phi_n$.

For $m \neq \l$, writing $\gamma_m:=\E^{(T_{m-1})}\Tr \Gl-\E^{(T_m)}\Tr \Gl$, we can easily check that
$$\gamma_m=\E^{(T_{m-1})}\sigma_m-\E^{(T_m)}\sigma_m, $$
where $\sigma_m:=\Tr \Gl-\Tr G^{(\l,m)}$. By (\ref{Interlacing}) we have $|\gamma_m| \leq C\eta^{-1}$. Hence we obtain
$$V_2 \leq C \left(\frac{1}{p^2}+\frac{|z|^2}{pn}\right)\left(\frac{p}{\eta^2}\right) \leq \frac{C_{\tau}}{p\eta^2}.$$

Plugging the estimates of $\E\Yl$ in statement (a), $V_1$ in (\ref{V1}) and $V_2$ above into the equation (\ref{Yl3}), we obtain the desired estimate of $\E|\Yl|^2$.
\end{proof}
\subsection{Proof of Theorem \ref{thm2}}
We can now complete the proof of Theorem \ref{thm2}.
\begin{proof}[Proof of Theorem \ref{thm2}]
We write (\ref{diagonal2}) as
\begin{eqnarray} \label{4:gll} G_{\l\l}=\frac{1}{\an+\Yl},\end{eqnarray}
where
$$\an=-z-s_n(z).$$
Taking expectations on both sides of (\ref{4:gll}), we can obtain
\begin{equation}\label{diagonal3}
\E G_{\l\l}=\frac{1}{\an}+\Al=\frac{1}{\an+\dl},
\end{equation}
where
\begin{equation}\label{Al}
\Al=\E\left(\frac{1}{\an+\Yl}\right)-\frac{1}{\an}=-\frac{1}{\an^2}\E \Yl+\frac{1}{\an^2}\E \left(\frac{\Yl^2}{\an+\Yl}\right),
\end{equation}
and
\begin{equation}\label{dl}
\dl=\left(\frac{1}{\an}+\Al\right)^{-1}-\an=-\frac{\an^2 \Al}{1+\an \Al}.
\end{equation}
For $\Al$, since
\begin{eqnarray*} \left|G_{\l\l}\right|=\left|\frac{1}{\an+\Yl}\right| \le \eta^{-1},\end{eqnarray*}
we obtain
\begin{equation}\label{Al2}
\left|\an^2 \Al \right|=\left|-\E \Yl+\E\frac{\Yl^2}{\an+\Yl}\right|\leq |\E \Yl|+\frac{1}{\eta}\E|\Yl|^2.
\end{equation}
For $\dl$, using the fact that $|\an| \ge \eta$ and Lemma \ref{Yl2} we obtain
\begin{eqnarray*} \label{4:dl} \left|\dl \right| \leq \frac{C_{\tau}}{\eta^3}\left(\frac{1}{p}+\sqrt\frac{p}{n}\right) \end{eqnarray*}
for any $z \in \st$.

Summing for all $\l \in [1\isep p]$ and then dividing $p$ on both sides of (\ref{diagonal3}), it is easy to see that in writing
\[s_n(z)= \frac{1}{\alpha_n+\Delta(z)},\]
the quantity $\Delta(z)$ satisfies the same bound as $\dl$ above. This completes the proof of Theorem \ref{thm2}.
\end{proof}


\section*{Acknowledgments}
The research of M. Xiong was supported by RGC grant number 16303615 from Hong Kong.

\end{document}